\pdfoutput=1
\documentclass[11pt]{amsart}

\usepackage[paper,pkg={amsaddr=true,tikz=true}]{zbs}

\newcommand{\Rhom}{R^{\mathrm{hom}}}

\title{An improved lower bound on the Shannon capacities of complements of odd cycles}
\author{Daniel G. Zhu}
\address{Department of Mathematics, Princeton University, Princeton, NJ 08544, USA}
\email{zhd@princeton.edu}

\begin{document}
\begin{abstract}
Improving a 2003 result of Bohman and Holzman, we show that for $n \geq 1$, the Shannon capacity of the complement of the $2n+1$-cycle is at least $(2^{r_n} + 1)^{1/r_n} = 2 + \Omega(2^{-r_n}/r_n)$, where $r_n = \exp(O((\log n)^2))$ is the number of partitions of $2(n-1)$ into powers of $2$. We also discuss a connection between this result and work by Day and Johnson in the context of graph Ramsey numbers.
\end{abstract}
\maketitle
\section{Introduction}
For a (finite, simple) graph $G$, let $\alpha(G)$ denote its independence number, $G^{\boxtimes k}$ denote its $k$th strong power, and $\Theta(G) = \lim_{k \to \infty} \alpha(G^{\boxtimes k})^{1/k} = \sup_{k\geq 1} \alpha(G^{\boxtimes k})^{1/k}$ denote its Shannon capacity. A longstanding problem in extremal combinatorics concerns the determination of the Shannon capacities of odd cycles $C_{2n+1}$ and their complements $\bar C_{2n+1}$. In some sense, these are the simplest graphs for which computing the Shannon capacity is nontrivial, for the following two reasons:
\begin{itemize}
\item The strong perfect graph theorem, proven by Chudnovsky, Robertson, Seymour, and Thomas \cite{SPGT}, states that every graph is either perfect, in which case its Shannon capacity is known to be equal to its independence number, or contains an odd cycle or a complement of an odd cycle, both on at least $5$ vertices, as an induced subgraph.
\item For complements of odd cycles, it is additionally true that, out of all graphs $G$ with at most $2n+2$ vertices, the graph $\bar C_{2n+1}$ has the smallest Shannon capacity that is not exactly $0$, $1$, or $2$. This can be argued directly as follows: it is easy to show that $\Theta(G) = 0$ if and only if $G$ is the empty graph, $\Theta(G) = 1$ if and only if $G$ is a nonempty complete graph, and $\Theta(G) = 2$ if $G$ is not complete but has a bipartite complement. For any other graph $G$ with at most $2n+2$ vertices, its complement contains an odd cycle, and since there exist graph homomorphisms $C_{2m+3} \to C_{2m+1}$ for all positive integers $m$, we conclude that there exists a graph homomorphism $C_{2n+1} \to \bar G$. It follows that $\Theta(G) \geq \Theta(\bar C_{2n+1}) > 2$, where we use the fact that the Shannon capacity respects the cohomomorphism preorder (see e.g.\ \cite{WZ23} for more on this connection), and the fact (to be discussed below) that $\Theta(\bar C_{2n+1}) > 2$.
\end{itemize}
While it was once speculated that $\Theta(\bar C_{2n+1})$ could be equal to $2$ for large $n$, this was disproved in 2003 by Bohman and Holzman \cite{BH03}, who showed that $\Theta(\bar C_{2n+1}) \geq (2^{2^{n-1}}+1)^{1/2^{n-1}}$. In contrast, the best known upper bounds come from the Lov\'asz theta function \cite{L79}:
\[\Theta(\bar C_{2n+1}) \leq \vartheta(\bar C_{2n+1}) = \frac{1 + \cos \frac{\pi}{2n+1}}{\cos \frac{\pi}{2n+1}} = 2 + \frac{\pi^2}{8} n^{-2} + o(n^{-2}).\]

In this paper, we prove the following result:
\begin{thm} \label{thm:main}
Let $r_n$ (sequence \href{https://oeis.org/A000123}{A000123} in \cite{OEIS}) be the number of partitions of $2(n-1)$ into powers of $2$. Then $\alpha(\bar C_{2n+1}^{\boxtimes r_n}) \geq 2^{r_n} + 1$, and in particular, $\Theta(\bar C_{2n+1}) \geq (2^{r_n} + 1)^{1/r_n} = 2 + \Omega(2^{-r_n}/r_n)$.
\end{thm}
Concretely, the first improvement on the existing literature appears at $n = 4$, where we show that $\Theta(\bar C_9) \geq 65^{1/6} \approx 2.00517$, as opposed to the previous bound of $257^{1/8} \approx 2.00097$. Asymptotically, as it is known that $\log r_n \sim (\log n)^2/(2 \log 2)$ \cite{M40}, this bound improves on the doubly exponential Bohman-Holzman bound but remains far away from any known upper bound.

\subsection*{Connection with graph Ramsey numbers}
After a preprint of this article was made public, the author was made aware of previous work of Day and Johnson \cite{DJ17} that constructed a graph coloring similar to the main construction of this article in the context of graph Ramsey numbers. To explain the relationship between these two contexts, we exhibit a connection between the Shannon capacity and multicolor graph Ramsey numbers, a straightforward extension of previous results of Erd\H{o}s, McEliece, and Taylor \cite{EMT71} (see also \cite{AO95}).

To state the connection, we recall that given graphs $H_1, H_2, \ldots, H_k$, the \vocab{graph Ramsey number} $R(H_1, H_2, \ldots, H_k)$ is the minimum $m$ such that in any $k$-edge-coloring of the complete graph $K_m$, one can, for some $i$, find a copy of $H_i$ with all edges of color $i$. In this context, however, it is more natural to work with what we will call a \vocab{homomorphic graph Ramsey number} $\Rhom(H_1, H_2, \ldots, H_k)$, which we define in the same way except that we now only need, for some $i$, a homomorphism from $H_i$ to the subgraph of $K_m$ consisting of the edges of color $i$. It is evident that we always have $\Rhom(H_1, H_2, \ldots, H_k) \leq R(H_1, H_2, \ldots, H_k)$, and that if all the $H_i$ are complete graphs $K_{r_i}$, both $\Rhom(H_1, H_2, \ldots, H_k)$ and $R(H_1, H_2, \ldots, H_k)$ are equal to the standard multicolor Ramsey number $R(r_1, r_2, \ldots, r_k)$. In the case where $H_1, H_2, \ldots, H_k$ are all equal to some graph $H$, further define $R_k(H) = R(H_1, H_2, \ldots, H_k)$ and $\Rhom_k(H) = \Rhom(H_1, H_2, \ldots, H_k)$.

We now make the following observation, which was stated in the case where the $H_i$ are all complete graphs by Erd\H{o}s, McEliece, and Taylor \cite{EMT71} and can be proven in the same way:
\begin{prop} \label{prop:rhomalpha}
If $H_1, H_2, \ldots, H_k$ are nonempty graphs, then $\Rhom(H_1, H_2, \ldots, H_k) - 1$ is the maximum of $\alpha(G_1 \boxtimes G_2 \boxtimes \cdots \boxtimes G_k)$ over tuples of graphs $(G_i)_{i \in [k]}$ such that there is no homomorphism $H_i \to \bar G_i$ for any $i \in [k]$. Moreover, if all the $H_i$ are equal and connected, the maximum can be achieved by taking all the $G_i$ to be equal.
\end{prop}
Now, by setting all the $H_i$ to be equal and considering the limit $k \to \infty$, we obtain a relation between the asymptotics of homomorphic graph Ramsey numbers and Shannon capacities. In the complete graph case, this fact was first observed by Alon and Orlitsky \cite{AO95}.
\begin{cor} \label{cor:ramseyshannon}
For a connected graph $H \neq K_1$, we have
\[\lim_{k \to \infty} \Rhom_k(H)^{1/k} = \sup_{k \geq 1}{(\Rhom_k(H)-1)^{1/k}} = \sup_{H \nrightarrow \bar G} \Theta(G),\]
where the second supremum is taken over all graphs $G$ for which there is no homomorphism $H \to \bar G$. This quantity can be infinite.
\end{cor}
For completeness, we include proofs of \cref{prop:rhomalpha,cor:ramseyshannon} in \hyperref[sec:appendix]{an appendix}.

Despite this equality, very little is known about the actual value of the quantity in \cref{cor:ramseyshannon}, which we will abbreviate as $c_H$. Indeed, whether we have $c_{K_3} = \infty$ is a famous open problem. In fact, we cannot even rule out the case where $c_H = \infty$ for all non-bipartite $H$ (for bipartite $H$ it is easy to see that $\Rhom_k(H) = 2$ and thus $c_H = 1$).

Using this language, the work of Day and Johnson can be reinterpreted as follows: for $n \geq 2$ a positive integer, they find a construction (equivalent to that described this paper) showing that $\alpha(\bar C_{2n_1+1} \boxtimes \bar C_{2n_2+1} \boxtimes \cdots \boxtimes \bar C_{2n_k+1}) > 2^k$ for some positive integers $k$ and $n_1, n_2, \ldots, n_k \geq n$. Then, applying what is effectively \cref{prop:rhomalpha} with $G_i = \bar C_{2n_i+1}$ and $H_i = C_{2n-1}$, they find $\Rhom_k(C_{2n-1}) > 2^k + 1$, which is equivalent to Theorem 2 in \cite{DJ17}. This in turn implies 
\[c_{C_{2n-1}} = \sup_{k \geq 1}{(\Rhom_k(C_{2n-1})-1)^{1/k}} > 2,\]
which is equivalent to Theorem 4 in \cite{DJ17}. An alternative way to proceed would have been to observe that their construction implies $\Theta(\bar C_{2n+1}) > 2$ after a short argument (basically \cref{prop:down} of this paper), and concluding $c_{C_{2n-1}} > 2$ from \cref{cor:ramseyshannon}.

There are many questions that remain in this area. Apart from computing $c_H$, one that stands out to the author is the following:
\begin{ques}
For a connected graph $H \neq K_1$, must we have
\[\lim_{k \to \infty} \Rhom_k(H)^{1/k} = \lim_{k \to \infty} R_k(H)^{1/k}?\]
\end{ques}
We note, however, that an affirmative answer could turn out to be uninteresting in the sense that both limits could be $1$ for all bipartite $H$ and $\infty$ for all other $H$, which would give relatively little information on the relationship between homomorphic and non-homomorphic graph Ramsey numbers.

\section{Proof of \texorpdfstring{\cref{thm:main}}{Theorem \ref{thm:main}}}
We let $[n] = \set{1,2,\ldots,n}$ as usual. We will let $x$ and $y$ denote vectors and $x_i$ and $y_i$ denote their components.

For positive integers $k, n_1, \ldots, n_k$, define the \vocab{box} $B(n_1,\ldots,n_k) = [2n_1] \times \cdots \times [2n_k]$. We will use exponents to denote multiplicity; for example, $B(2^2, 3) = B(2, 2, 3)$. Let $\partial B(n_1,\ldots,n_k)$ be the elements $x \in B(n_1,\ldots,n_k)$ such that $x_i \in \set{1,2n_i}$ for some $i \in [k]$. Furthermore, for $x, y \in B(n_1,\ldots,n_k)$ write $x \sim y$ if $x = y$ or $\abs{x_i - y_i}= 1$ for some $i \in [k]$. Finally, call a box $B(n_1,\ldots,n_k)$ \vocab{good} if there is a subset $S \subseteq \partial B(n_1,\ldots,n_k)$ (which we will call a \vocab{skeleton}) of size $2^k$ such that $x \sim y$ for all $x, y \in S$. Note that whether a box is good is independent of the order of the $n_i$, so in the remainder of the paper we will identify boxes with the same dimensions but in a different order.

\begin{rmk}
If $x \neq y$ and $x \sim y$, then $x$ and $y$ cannot be equal modulo $2$. Thus, if $2^k$ is replaced with $2^k+1$ in the definition of a good box, no box can be good.
\end{rmk}

\begin{prop} \label{prop:goodtobound}
If $B(n^k)$ is good, then $\alpha(\bar C_{2n+1}^{\boxtimes k}) \geq 2^k + 1$.
\end{prop}
\begin{proof}
Identify the vertices of $\bar C_{2n+1}$ with $\setz/(2n+1)\setz$. Then, if $S$ is a skeleton, we claim that $S \cup \set{0}$ is an independent set of $\bar C_{2n+1}^{\boxtimes k}$. Indeed, if $x,y \in B(n^k)$ are such that $x \sim y$, then there is no edge between $x$ and $y$ in $\bar C_{2n+1}^{\boxtimes k}$. Additionally, if $x \in \partial B(n^k)$, then there is no edge between $x$ and $0$ in $\bar C_{2n+1}^{\boxtimes k}$.
\end{proof}

Thus it suffices to show that $B(n^{r_n})$ is good, which we will achieve by describing a general framework to show that various boxes are good. First, we show that the set of good boxes is ``downwards closed''.
\begin{prop} \label{prop:down}
If the box $B(n_1,n_2,\ldots,n_k)$ is good and $n_1 > 1$, then $B(n_1 - 1,n_2,\ldots,n_k)$ is good.
\end{prop}
\begin{proof}
Consider the function $\phi \colon B(n_1,n_2,\ldots,n_k) \to B(n_1 - 1,n_2,\ldots,n_k)$ defined as follows:
\[\phi(x_1,x_2,\ldots,x_k) = \begin{cases}
(x_1,x_2,\ldots,x_k) & x_1 \leq 2(n_1 - 1) \\
(x_1-2,x_2,\ldots,x_k) & x_1 > 2(n_1 - 1). \\
\end{cases}\]
We claim that if $S$ is a skeleton of $B(n_1,n_2,\ldots,n_k)$, then $\phi(S)$ is a skeleton of $B(n_1-1,n_2,\ldots,n_k)$. To see this, note that $\phi$ sends $\partial B(n_1,n_2,\ldots,n_k)$ to $\partial B(n_1-1,n_2,\ldots,n_k)$, so $\phi(S) \subseteq \partial B(n_1-1,n_2,\ldots,n_k)$. Moreover, if $x\neq y$ and $\phi(x) = \phi(y)$, we must have $\abs{x_1-y_1} = 2$ and $x_i = y_i$ for all $1 < i \leq k$, implying that $x \not\sim y$. Thus, $\phi$ is injective on $S$, meaning that $\abs{\phi(S)} = 2^k$. Finally, consider distinct $x, y \in S$. If $\abs{x_i - y_i} = 1$ for some $i > 1$, we must have $\phi(x)\sim \phi(y)$ since $\phi$ leaves the $i$th coordinates unchanged. If $\abs{x_1 - y_1} = 1$, it is straightforward to verify that the first coordinates of $\phi(x)$ and $\phi(y)$ differ by $1$ as well, so we still have $\phi(x) \sim \phi(y)$.
\end{proof}

Now we describe a method for obtaining larger good boxes from the obviously good $B(1)$.

\begin{defn}
For positive integers $a$, $b$, and $c$, an \vocab{$(a; b, c)$-expansion} is a injective function $\psi \colon B(a) \times [2] \to B(b, c)$ mapping $\partial B(a) \times [2]$ to $\partial B(b, c)$ and such that for all $x, y \in B(a)$ with $x \sim y$ and $i, j \in [2]$, we have $\psi(x, i) \sim \psi(y, j)$.
\end{defn}
\begin{prop} \label{prop:appex}
If the box $B(n_1, \ldots, n_k)$ is good and an $(n_k; n', n'')$-expansion exists, then the box $B(n_1, \ldots, n_{k-1}, n', n'')$ is good.
\end{prop}
\begin{proof}
If $S$ is a skeleton of $B(n_1,\ldots,n_k)$ and $\psi$ is an $(n_k;n',n'')$ expansion, we claim that $S' = (\id^{k-1} \times \psi)(S \times [2])$ is a skeleton of $B(n_1,\ldots,n_{k-1},n',n'')$. For the remainder of the proof we will write $x,y \in S$ as $(x',x_k)$ and $(y',y_k)$, where $x',y' \in B(n_1,\ldots,n_{k-1})$.

For a given $x \in S$ and $i \in [2]$, note that we must have $x' \in \partial B(n_1,\ldots,n_{k-1})$ or $x_k \in \partial B(n_k)$. In either case, it follows that $(x', \psi(x_k, i)) \in \partial B(n_1,\ldots,n_{k-1},n',n'')$, since $\psi$ maps $\partial B(n_k) \times [2]$ to $\partial B(n',n'')$. The fact that $\abs{S'} = 2^{k+1}$ follows from the injectivity of $\psi$. Finally, for distinct $(x,i),(y,j) \in S \times [2]$, we must have $x' \neq y'$ and $x' \sim y'$, or $x_k \sim y_k$ and $(x_k,i) \neq (y_k,j)$. In both cases, we have $(x', \psi(x_k, i)) \sim (y', \psi(x_k, j))$, since $x_k \sim y_k$ implies $\psi(x_k,i) \sim \psi(y_k,j)$.
\end{proof}
\begin{prop} \label{prop:ex}
For $n \geq 1$, an $(n; n+1, 2n)$-expansion exists.
\end{prop}
\begin{proof}
The expansion is given by $\psi(x, 1) = (x, 2x)$ and $\psi(x, 2) = (x+2, 2x-1)$, which can be easily verified to satisfy the conditions. For an illustration, see \cref{ex1}.
\end{proof}
\begin{figure}[tbp]\centering
\subcaptionbox{\label{ex1}}{%
\begin{tikzpicture}[scale=0.65,every node/.style={draw,circle,inner sep=0.5mm}]
\foreach \x in {1,...,8}
\foreach \y in {1,...,12}
{\node at (\x,\y) {};}
\begin{scope}[every node/.append style={fill}]
\newcommand{\kfour}{(3,1) node {} -- (4,3) node {} -- (2,4) node {} -- (1,2) node {} -- (3,1) -- (2,4) (4,3)--(1,2)}
\draw \kfour;
\draw[shift={(1,2)}] \kfour;
\draw[shift={(2,4)}] \kfour;
\draw[shift={(3,6)}] \kfour;
\draw[shift={(4,8)}] \kfour;
\end{scope}
\end{tikzpicture}}\hspace{4em}\subcaptionbox{\label{ex2}}{%
\begin{tikzpicture}[scale=0.65,every node/.style={draw,circle,inner sep=0.5mm}]
\foreach \x in {1,...,8}
\foreach \y in {1,...,8}
{\node at (\x,\y) {};}
\begin{scope}[every node/.append style={fill}]
\newcommand{\kfour}{(3,1) node {} -- (4,3) node {} -- (2,4) node {} -- (1,2) node {} -- (3,1) -- (2,4) (4,3)--(1,2)}
\draw[shift={(2,2)}] \kfour;
\draw (2,4) node {}--(7,5) node {}--(6,5)--(2,4)--(3,4)--(7,5);
\draw (1,4) node {}--(8,5) node {}--(7,5)--(1,4)--(2,4)--(8,5);
\draw (5,2) node {}--(4,7) node {}--(4,6)--(5,2)--(5,3)--(4,7);
\draw (5,1) node {}--(4,8) node {}--(4,7)--(5,1)--(5,2)--(4,8);
\end{scope}
\end{tikzpicture}}
\caption{\subref{ex1} The $(3;4,6)$-expansion $\psi$ defined in the proof of \cref{prop:ex}. Filled circles represent the range of $\psi$ inside $B(4,6)$, and an edge is drawn between $\psi(x,i)$ and $\psi(y,j)$ for all $x,y\in B(3)$ and $i,j\in [2]$ with $x \sim y$ and $(x,i) \neq (y,j)$.\quad\subref{ex2} The $(3;4,4)$-expansion defined in \cref{rmk:bh}, depicted similarly to \subref{ex1}.}
\end{figure}
\begin{proof}[Proof of \cref{thm:main}]
Let $a_n$ be the number of partitions of $n$ into powers of $2$. We have $a_0 = 1$, $a_n = a_{n-1}$ if $n \geq 1$ is odd, and $a_n = a_{n-1} + a_{n/2}$ if $n \geq 2$ is even.\footnote{For a proof, note that if $n\geq 1$, the number of partitions of $n$ into powers of $2$ that include a $1$ is given by $a_{n-1}$, so $a_n - a_{n-1}$ is the number of partitions of $n$ into powers of $2$ that are all at least $2$, which is $0$ if $n$ is odd and $a_{n/2}$ if $n$ is even.} By definition, $r_n = a_{2(n-1)}$.

Consider starting with $B(1)$ and repeatedly applying \cref{prop:appex,prop:ex} to the smallest dimension, creating an infinite series of good boxes starting with $B(1)$, $B(2,2)$, $B(2,3,4)$, and $B(3,3,4,4)$. If we let $\bm{B}_n$ be the first box in this sequence where all dimensions are at least $n$, it is straightforward to show by induction that
\[\bm{B}_n = \begin{cases}
B(1)=B(1^{a_1}) & n=1,\\
B(2,2)=B(2^{a_2}) & n=2, \\
B((2m-1)^{a_{2m-1}}, (2m)^{a_{m}}, (2(m+1))^{a_{m+1}}, \ldots, (2(2m-2))^{a_{2m-2}}) & n =2m-1\geq 3 \text{ odd}, \\
B((2m)^{a_{2m}}, (2(m+1))^{a_{m+1}}, (2(m+2))^{a_{m+2}}, \ldots, (2(2m-1))^{a_{2m-1}}) & n = 2m \geq 4 \text{ even}.
\end{cases}\]
Thus, by \cref{prop:down} we find that $B(n^k)$ is good, where
\[k = a_n + \sum_{i = \floor{n/2} + 1}^{n-1} a_i = a_n + \sum_{i = \floor{n/2} + 1}^{n-1} (a_{2i} - a_{2i-2}) = a_n + a_{2(n-1)} - a_{2\floor{n/2}} = r_n.\]
Applying \cref{prop:goodtobound} completes the proof.
\end{proof}
\begin{rmk} \label{rmk:bh}
The language of good boxes and expansions can also be applied to the construction given in \cite{BH03}. In particular, there is a family of $(n; n+1, n+1)$-expansions given by
\[\psi(x, i) = \begin{cases*}
(x, n+1) & $x \leq n$ and $i = 1$, \\
(n+1, x+2) & $x > n$ and $i = 1$, \\
(2n+3-x, n+2) & $x \leq n$ and $i = 2$, \\
(n+2, 2n+1-x) & $x > n$ and $i = 2$. \\
\end{cases*}\]
(The $n = 3$ case is shown in \cref{ex2}.) Using these expansions, repeatedly applying \cref{prop:appex} to $B(1)$ yields a skeleton in the box $B(n^{2^{n-1}})$ for all positive integers $n$, which can be shown to be equivalent to the independent set constructed in \cite{BH03}.
\end{rmk}

\section{Further Remarks Concerning Good Boxes}
While a box $B(n^k)$ being good immediately implies a lower bound on $\Theta(\bar C_{2n+1})$, the converse is not necessarily true: the independent sets produced by good boxes contain a single distinguished point (namely $0$) that is different from every other point in every coordinate, and there is little reason to expect this to be true in an optimal configuration. Nonetheless, the question of which boxes are good is still interesting, and in this section we collect some results.
\begin{prop} \label{prop:adddim}
If $B(n_1,\ldots,n_k)$ is good, then $B(n_1,\ldots,n_k, n_{k+1})$ is good for any positive integer $n_{k+1}$.
\end{prop}
\begin{proof}
If $S$ is a skeleton of $B(n_1,\ldots,n_k)$, then it is straightforward to show that $S \times \set{1,2}$ is a skeleton of $B(n_1,\ldots,n_k,n_{k+1})$.
\end{proof}
A partial converse is also true:
\begin{prop} \label{prop:removedim}
If $B(n_1,\ldots,n_k,n_{k+1})$ is good and $n_{k+1} > 2^k$, then $B(n_1,\ldots,n_k)$ is good.
\end{prop}
\begin{proof}
Let $S$ be a skeleton of $B(n_1,\ldots,n_k,n_{k+1})$. Since $n_{k+1} > 2^k$, there exists some $c \in [2n_{k+1}]$ such that no element $x \in S$ satisfies $x_{k+1} = c$. Let $A \subseteq [2n_{k+1}]$ consist of all even numbers less than $c$ and all odd numbers greater than $c$. Then we claim that
\[S' = \setmid{(x_1,\dots,x_k) \in B(n_1,\ldots,n_k)}{(x_1,\ldots,x_k,x_{k+1}) \in S\text{ for some }x_{k+1} \in A}\]
is a skeleton of $B(n_1,\ldots,n_k)$.

To see this, note that since $1$ and $2n_{k+1}$ are not in $A$, the set $S \cap (B(n_1,\ldots,n_k) \times A)$ is contained within $\partial B(n_1,\ldots,n_k) \times A$, implying that $S' \subseteq \partial B(n_1,\ldots,n_k)$. Also, since no two elements of $A$ differ by $1$, we must have $x \sim y$ for all $x,y \in S'$.

It remains to show that $\abs{S'} = 2^k$. To accomplish this, observe that since all the elements of $S$ are distinct modulo $2$, we can define an involution $f \colon S \to S$ such that if $y = f(x)$, the quantity $x_i - y_i$ is odd if and only if $i = k+1$. In this case, since $x \sim y$, we must in fact have $\abs{x_{k+1} - y_{k+1}} = 1$. Moreover, since neither $x_{k+1}$ nor $y_{k+1}$ is equal to $c$, exactly one of $x_{k+1}$ and $y_{k+1}$ is in $A$. Thus,
\[\abs{S \cap (B(n_1,\ldots,n_k) \times A)} = \frac{\abs{S}}{2} = 2^k.\]
Finally, it is impossible for distinct $x,y \in S \cap (B(n_1,\ldots,n_k) \times A)$ to project to the same point in $S'$, since it would then be impossible for $x \sim y$ to hold. Thus $\abs{S'} = 2^k$ as well.
\end{proof}
We remark that the constant $2^k$ is tight. To see this, note that by iterating \cref{prop:appex,prop:ex}, the box $B(2^0+1,\ldots,2^{k-1}+1,2^k)$ is good. On the other hand, by iterating \cref{prop:removedim}, the box $B(2^0+1,\ldots,2^{k-1}+1)$ is not good.

We conclude with a complete listing of good boxes in dimension $k \leq 4$:
\begin{itemize}
\item If $k=1$: $B(1)$.
\item If $k=2$: $B(1, \infty)$, $B(2,2)$, and any smaller box (in the sense of \cref{prop:down}), where $\infty$ represents an arbitrarily large positive integer.
\item If $k=3$: $B(1, \infty, \infty)$, $B(2, 2, \infty)$, $B(2,3,4)$, and any smaller box.
\item If $k=4$: $B(1,\infty,\infty,\infty)$, $B(2,2,\infty,\infty)$, $B(2,3,4,\infty)$, $B(2,3,5,8)$, $B(2,4,4,6)$, $B(3,3,3,5)$, $B(3,3,4,4)$, and any smaller box.
\end{itemize}
The cases $k = 1,2$ above are either trivial or follow from \cref{prop:removedim}, while the cases $k=3,4$ require computer assistance (for more details see \url{https://osf.io/qdfa6/}). Except for $B(3,3,3,5)$, all constructions follow from \cref{prop:appex,prop:ex,prop:adddim}. For the purposes of generating large good boxes using \cref{prop:appex,prop:ex}, the box $B(3,3,3,5)$ performs strictly worse than $B(3,3,4,4)$ and thus does not yield better bounds for $\Theta(\bar C_{2n+1})$.

\section*{Appendix.\texorpdfstring{\enspace}{ }Proofs of \texorpdfstring{\cref{prop:rhomalpha,cor:ramseyshannon}}{Theorem \ref{prop:rhomalpha} and Corollary \ref{cor:ramseyshannon}}} \label{sec:appendix}
As remarked in the introduction, these proofs are minor modifications of those in \cite{EMT71, AO95}.
\subsection*{Proof of \texorpdfstring{\cref{prop:rhomalpha}}{Theorem \ref{prop:rhomalpha}}}
Since $H_1, H_2, \ldots, H_k$ are nonempty, $\Rhom(H_1, H_2,\ldots,H_k)-1$ is the maximum $m$ for which we can partition the edges of $K_m$ into graphs $G'_1, G'_2, \ldots, G'_k$ (all on the same vertex set as $K_m$) such that there does not exist a homomorphism $H_i \to G'_i$ for any $i$. Note that it is equivalent to allow the $G'_i$ to share edges, since given such a decomposition, we can delete edges until the $G'_i$ are edge-disjoint.

Given graphs $G_1, G_2, \ldots, G_k$ such that there is no homomorphism $H_i \to \bar G_i$ for any $i$, let $G'_i$ be the $\prod_{j \neq i} V(G_j)$-blowup of $\bar G_i$, i.e.\ the graph on vertex set $\prod_{i \in [k]} V(G_i)$ such that $(u_1, u_2, \ldots, u_k)$ is adjacent to $(v_1, v_2, \ldots, v_k)$ if and only if $u_iv_i \in E(\bar G_i)$. Since projection onto the $i$th coordinate yields a homomorphism $G'_i \to \bar G_i$, there is no homomorphism $H_i \to G'_i$ for any $i$. Also, we have $\bigcup_{i\in[k]} G'_i = \overline{G_1 \boxtimes G_2 \boxtimes \cdots \boxtimes G_k}$. Now, letting $m = \alpha(G_1 \boxtimes G_2 \boxtimes \cdots \boxtimes G_k)$, we find that $K_m$ is a subgraph of $\overline{G_1 \boxtimes G_2 \boxtimes \cdots \boxtimes G_k}$, so by taking induced subgraphs of the $G'_i$ we achieve the desired decomposition of $K_m$ and find that $m \leq \Rhom(H_1, H_2,\ldots,H_k)-1$.

Conversely, letting $m = \Rhom(H_1, H_2,\ldots,H_k)-1$ and taking a decomposition $K_m = \bigcup_{i \in [k]} G'_i$, it is straightforward to show that the diagonal is an independent set in $\bar G'_1 \boxtimes \bar G'_2 \boxtimes \cdots \boxtimes \bar G'_k$. Therefore, letting $G_i = \bar G'_i$, we find that there are no homomorphisms $H_i \to \bar G_i$ and $\alpha(G_1 \boxtimes G_2 \boxtimes \cdots \boxtimes G_k) \geq m$.

Finally, suppose all the $H_i$ are a connected graph $H$. Given graphs $G_1, G_2, \ldots, G_k$ with no homomorphisms $H \to \bar G_i$, let $G = \overline{\bar G_1 \sqcup \bar G_2 \sqcup \cdots \sqcup \bar G_k}$. Since $H$ is connected, there is no homomorphism $H \to \bar G$. Also, since $G_i$ is an induced subgraph of $G$ for all $i$, we find that $G_1 \boxtimes G_2 \boxtimes \cdots \boxtimes G_k$ is an induced subgraph of $G^{\boxtimes k}$. Therefore $\alpha(G_1 \boxtimes G_2 \boxtimes \cdots \boxtimes G_k) \leq \alpha(G^{\boxtimes k})$, meaning that the maximum can be achieved with all the $G_i$ equal.

\subsection*{Proof of \texorpdfstring{\cref{cor:ramseyshannon}}{Corollary \ref{cor:ramseyshannon}}}
First note that since $H$ is not edgeless, $\Rhom_k(H) \geq 2$, so
\[\lim_{k \to \infty} \Rhom_k(H)^{1/k} = \lim_{k\to\infty} (\Rhom_k(H)-1)^{1/k}.\]
We will now prove
\[\limsup_{k\to\infty} {(\Rhom_k(H)-1)^{1/k}} \overset{(1)}{\leq} \sup_{k\geq 1} {(\Rhom_k(H)-1)^{1/k}} \overset{(2)}{=} {\sup_{H \nrightarrow \bar G} \Theta(G)} \overset{(3)}{\leq} \liminf_{k\to\infty} {(\Rhom_k(H)-1)^{1/k}}.\]
Statement (1) is a basic property of limits. To prove (2), note that by \cref{prop:rhomalpha},
\[\sup_{k\geq 1} {(\Rhom_k(H)-1)^{1/k}} = \sup_{\substack{k\geq 1 \\ H \nrightarrow \bar G}} {\alpha(G^{\boxtimes k})^{1/k}} = \sup_{H \nrightarrow \bar G} \Theta(G).\]
Finally, (3) follows from the fact that for any graph $G$ with no homomorphism $H \to \bar G$, we have by \cref{prop:rhomalpha} that
\[\Theta(G) = \liminf_{k \to \infty} \alpha(G^{\boxtimes k})^{1/k} \leq \liminf_{k \to \infty} {(\Rhom_k(H)-1)^{1/k}}.\]

\section*{Acknowledgments}
This work was supported by a Princeton First-Year Fellowship and by the NSF Graduate Research Fellowships Program (grant number: DGE-2039656). The author would like to thank Noga Alon, Zach Hunter, Cyril Pujol, Varun Sivashankar, and two anonymous referees for helpful comments and discussions.

\printbibliography
\end{document}